\documentclass[12pt,oneside]{amsart}

\usepackage{amssymb,latexsym}

\usepackage{hyperref}
\hypersetup{
  colorlinks   = true, 
  urlcolor     = blue, 
  linkcolor    = blue, 
  citecolor   = red 
}
\usepackage{anysize}
\usepackage{subcaption}

\usepackage{pgfplots}
\usepackage{mathrsfs}
\usetikzlibrary{arrows}
\usetikzlibrary{calc}
\usetikzlibrary{intersections}
\usetikzlibrary{through}

\usepackage{mathtools}

\newtheorem{theorem}{Theorem}[section]
\newtheorem{proposition}[theorem]{Proposition}
\newtheorem{lemma}[theorem]{Lemma}

\theoremstyle{definition}
\newtheorem{definition}[theorem]{Definition}

\theoremstyle{remark}
\newtheorem{remark}[theorem]{Remark}

\numberwithin{equation}{section}

\renewcommand{\epsilon}{\varepsilon}
\renewcommand{\phi}{\varphi}
\renewcommand{\kappa}{\varkappa}
\newcommand{\dist}{\text{dist}}
\newcommand{\pdist}{\text{pdist}}
\newcommand{\per}{\text{per}}
\newcommand{\pper}{\text{pper}}
\newcommand{\diam}{\text{diam}}

\begin{document}

\title{On the total perimeter of disjoint convex bodies}

\author{Arseniy Akopyan and Alexey Glazyrin}

\address{Arseniy Akopyan, FORA Capital, Miami, FL, USA}
\email{akopjan@gmail.com}

\address{Alexey Glazyrin, School of Mathematical \& Statistical Sciences, The University of Texas Rio Grande Valley, Brownsville, TX 78520, USA}
\email{alexey.glazyrin@utrgv.edu}


\begin{abstract}
In this note we introduce a pseudometric on convex planar curves based on distances between normal lines and show its basic properties. Then we use this pseudometric to give a short proof of the theorem by Pinchasi that the sum of perimeters of $k$ convex planar bodies with disjoint interiors contained in a convex body of perimeter $p$ and diameter $d$ is not greater than $p+2(k-1)d$.
\end{abstract}

\maketitle

\section{Introduction}

For a convex body $C$, we denote its perimeter by $\per(C)$ and its diameter by $\diam(C)$. Given a convex body $C$, it is natural to find the maximal total perimeter of $k$ disjoint convex bodies confined to $C$. Glazyrin and Mori\'{c} studied this problem in \cite{GM14} and conjectured that the upper bound is always $\per(C) + 2(k-1)\diam(C)$. They proved this bound for some particular cases and made partial progress towards the general conjecture by showing the upper bound $1.22195\per(C) + 2(k-1)\diam(C)$. In \cite{P17}, Pinchasi proved the general conjecture.

\begin{theorem}\label{thm:main}[Pinchasi]
If convex planar bodies $C_i$, $1\leq i\leq k$, with disjoint interiors are contained in a convex planar body $C$, then
$$\sum\limits_{i=1}^k \per(C_i)\leq \per(C) + 2(k-1)\diam(C).$$
\end{theorem}

In this note we provide a short and (almost) self-contained proof of Theorem \ref{thm:main} using the construction of a pseudometric on convex curves. Apart from the proof, we find this pseudometric and its properties inherently interesting. For the sake of simplicity, some statements will be formulated for \textit{strictly convex curves}, that is, curves that have exactly one common point with each supporting line. In particular, all curves of constant widths are strictly convex curves (see, for instance, \cite[Theorem 3.1.1]{MMO19}). All statements are true for general convex curves as well, with minor modifications. Throughout the whole paper by convex curves we always mean closed convex curves and use the term almost interchangeably with convex bodies.

\section{Pseudometric on convex curves}

In this section we define the pseudometric on the set of convex planar curves and adjacent notions. We fix a unit vector $u_0$ in the plane and by $u_{\theta}$, $\theta\in[0,2\pi]$, define a unit vector that is obtained by rotating $u_0$ by angle $\theta$ counterclockwise. Given a convex curve $C$, for each $u_{\theta}$ there is a unique supporting line to $C$ such that $u_{\theta}$ is orthogonal to it and the direction of $u_{\theta}$ corresponds to the half-plane containing $C$. For a common point of the supporting line and $C$, we will say that $u_{\theta}$ is an inward normal vector to $C$ at this point. By a normal line $\ell_{\theta}(C)$ we denote the line through the common point of the supporting line and $C$ in the direction of $u_{\theta}$. Note that normal lines are uniquely defined for all $\theta$ when $C$ is strictly convex. For general convex curves, $\ell_{\theta}$ are uniquely defined for all but countably many values of $\theta$ (this follows, for instance, from the result in \cite{AK52}). Therefore, the integrals in the definitions below are properly defined for all pairs of convex curves.

\begin{definition}\label{def:pdist}
For two convex curves $C_1$ and $C_2$, we define
$$\pdist(C_1,C_2)=\frac 1 2\int\limits_0^{2\pi} \dist(\ell_{\theta}(C_1), \ell_{\theta}(C_2)) d\theta,$$
where $\dist$ is a standard Euclidean distance between parallel lines.
\end{definition}

\begin{remark}
It is easy to show that if instead of $\dist$ we take the signed distance between lines, then the integral in the right hand side is zero. 
\end{remark}

\begin{proposition}\label{prop:pseudometric}
The space of all convex curves equipped with the distance from Definition \ref{def:pdist} is a pseudometric space. Moreover, the distance is convex with respect to Minkowski addition, that is,
$$\pdist(tC_1+(1-t)C_2, D)\leq t\, \pdist(C_1, D) + (1-t)\, \pdist(C_2, D),$$
for all convex curves $C_1$, $C_2$, $D$ and for any $t\in[0,1]$. The equality for $t\in(0,1)$ holds only if $\ell_{\theta}(C_1)$ and $\ell_{\theta}(C_2)$ are on the same side of $\ell_{\theta}(D)$ for all $\theta$.
\end{proposition}

\begin{proof}
The triangle inequality follows immediately from the one-dimensional triangle inequality for each $\theta$, that is, from $$\dist(\ell_{\theta}(C_1), \ell_{\theta}(C_3)) \leq \dist(\ell_{\theta}(C_1), \ell_{\theta}(C_2)) + \dist(\ell_{\theta}(C_2), \ell_{\theta}(C_3)).$$
The remaining properties of a pseudometric follow trivially from the definition.

For convexity, note that the normal line of a Minkowski sum is the Minkowski sum of normal lines of the summands:
$$\ell_{\theta}(tC_1+(1-t)C_2)=t \ell_{\theta}(C_1) + (1-t)\ell_{\theta}(C_2).$$
Now it remains to use the fact that for each $\theta$ the distance to a given line $\ell_{\theta}(D)$ is a convex function (in other terms, for a fixed real $a$, the real function $|x-a|$ is convex). If $\ell_{\theta}(C_1)$ and $\ell_{\theta}(C_2)$ are in different open half-planes defined by $\ell_{\theta}(D)$ then, due to continuity, there is an interval of $\theta$ where this holds and the inequality must be strict.
\end{proof}

A single point can be considered a degenerate strictly convex curve. Definition \ref{def:pdist} and Proposition \ref{prop:pseudometric} work in this case just the same. Moreover, for two points $v_1$ and $v_2$, $\pdist(v_1,v_2)$ coincides with the regular Euclidean distance multiplied by $2$.

The distance from Definition \ref{def:pdist} does not define a metric because there are curves that are different but have the same normal line bundles, for example, concentric circles. Similarly, we can generalize a perimeter of a convex curve. 

\begin{definition}\label{def:pper}
For convex curves $C$ and $D$,
$$\pper_D(C)=\frac 1 2\int\limits_0^{2\pi} (\dist(\ell_{\theta}^r(C), \ell_{\theta}(D))+ \dist(\ell_{\theta}^l(C), \ell_{\theta}(D)) )d\theta,$$
where $\ell_{\theta}^{r}(C)$ and $\ell_{\theta}^l(C)$ are supporting lines to $C$ parallel to $\ell_{\theta}(D)$.
\end{definition}

Note that the standard perimeter (curve length) $\per(C)$ can be obtained by a similar formula with $\dist(\ell_{\theta}^r(C),\ell_{\theta}^l(C))$ under the integral. Therefore, the triangle inequality implies that $\pper_D(C)\geq \per(C)$ and the equality holds if and only if each normal line of $D$ intersects $C$. This happens, for instance, when $D$ is a disk containing $C$ and the center of $D$ is inside $C$ so $\pper$ is indeed a generalization of the standard Euclidean perimeter.

It immediately follows from the definitions that for any point $v$ and any curve $D$, $$\pper_D(v) = 2 \pdist(v,D).$$

For curves of constant width, the definitions above become even more convenient. In particular, $\ell_{\theta}=\ell_{\theta+\pi}$ for each $\theta\in[0,\pi]$ so we can substitute $\frac 1 2 \int_0^{2\pi}$ by $\int_{0}^{\pi}$ both in Definition~\ref{def:pdist} and Definition \ref{def:pper}.

The following lemma essentially extends the definition of a circle to all curves of constant widths using the distance from Definition \ref{def:pdist}.

\begin{lemma}\label{lem:pdist_diam}
For a constant width curve $D$ and a point $v$, $\pdist(v,D)=\diam(D)$ for $v\in D$, $\pdist(v,D)< \diam(D)$ for $v$ inside $D$, and $\pdist(v,D)> \diam(D)$ for $v$ outside~$D$.
\end{lemma}

In order to prove this lemma, we will need the following result about convex curves. Assume $\gamma$ is a strictly convex curve parametrized by inward normal vectors, that is, $\gamma(\theta)$ is a point on the curve, where the inward normal vector is $u_{\theta}$ ($\gamma$ is not injective for singular points of the curve).
We do not need a precise form of this parametrization but the mere fact that $\gamma(\theta)$ is almost everywhere differentiable (this follows from \cite[Theorem 5.1.1]{MMO19} or from the theorem of Aleksandrov \cite{A39}). For each $\theta$, we define a unit tangent vector $v_{\theta}$ by rotating $u_{\theta}$ by $\pi/2$ clockwise.

\begin{lemma}\label{lem:signed}
For all strictly convex curves $\gamma$ and for any pair of $\theta_1$, $\theta_2$ such that $0\leq\theta_1\leq\theta_2\leq 2\pi$,
$$\int\limits_{\theta_1}^{\theta_2} \gamma(\theta) \cdot v_{\theta}\, d\theta = \gamma(\theta_1)\cdot u_{\theta_1} - \gamma(\theta_2) \cdot u_{\theta_2},$$
where $\cdot$ is the standard dot product in the plane.
\end{lemma}

\begin{proof}
Differentiating $\gamma(\theta)\cdot u_{\theta}$ we get $\dot{\gamma}\cdot u_{\theta} + \gamma\cdot (-v_{\theta}) = - \gamma\cdot v_{\theta}$ for almost all $\theta$ so the integral at hand is just the difference of values of $\gamma\cdot v_{\theta}$ at $\theta_1$ and $\theta_2$.
\end{proof}

\begin{remark}
This lemma may be extended to general convex curves by taking into account line segments in $D$ or carefully approximating $D$ with strictly convex curves.
\end{remark}

We also note that all curves of constant width are naturally characterized by their parametrizations using the support function \cite{HS53I}, \cite{HS53II}, \cite[Theorem 5.3.5]{MMO19}.

The dot product $\gamma(\theta)\cdot v_{\theta}$ is precisely the distance from the origin 0 to the normal line $\ell_{\theta}$ taken with a sign. The integral calculated in Lemma \ref{lem:signed} can be used to calculate a part of $\pdist(0,\gamma)$ as $|\gamma(\theta_1)\cdot u_{\theta_1} - \gamma(\theta_2) \cdot u_{\theta_2}|$  when the origin lies on the same side of all normal lines $\ell_{\theta}$ for all $\theta\in[\theta_1,\theta_2]$. This is exactly what we are going to do in the proof of Lemma \ref{lem:pdist_diam}.

\begin{proof}[Proof of Lemma \ref{lem:pdist_diam}]
Let $v\in D$ and assume that $u_{\theta_1}$ is the inward normal vector to $D$ at $v$. Without loss of generality, $v=0$ and $\theta_1=0$. Parametrizing the curve as above, $v=\gamma(0)$. Let $v'=\gamma(\pi)$. Note that $v'$ is diametrically opposite to $v$ on $D$.

Clearly, $v$ lies on the same side of all $\ell_{\theta}$ for $\theta\in(0, \pi)$ ($v$ may belong to some of these lines if it is a singular point of the curve). By Lemma \ref{lem:signed}, 
$$\pdist(v,D)=\int\limits_0^{\pi} \dist(v,\ell_\theta)\, d\theta = |\gamma(0)\cdot u_0 - \gamma(\pi) \cdot u_{\pi}| = |\gamma(\pi)| = \diam(D).$$

Let $v$ be inside $D$. Assume $v_1$ is the point in $D$ closest to $v$. Then the line $v_1v$ is a normal line of $D$. Assume it intersects $D$ at the second point $v_2$. $v_1$ and $v_2$ are diametrically opposite points of $D$. There is a normal line $\ell_{\theta}(D)$ such that $v_1$ and $v_2$ are in different open half-planes with respect to this line (for example, when $u_{\theta}$ is orthogonal to the line $v_1v_2$). By the convexity part of Proposition \ref{prop:pseudometric}, $\pdist(v,D)$ is strictly smaller than $t\, \pdist(v_1,D)+(1-t)\, \pdist(v_2,D)$ for a certain $t\in (0,1)$. Since $\pdist(v_1,D)=\pdist(v_2,D)=\diam(D)$, $\pdist(v,D)<\diam(D)$.

Let $v$ be outside $D$. Connect it to any point $v_1$ inside $D$. Assume the line segment $vv_1$ intersects $D$ at $v_2$. Then $\pdist(v_2,D)=\diam(D)$ and $\pdist(v_1,D)<\diam(D)$. By convexity, $\pdist(v,D)>\diam(D)$.
\end{proof}

In the last lemma of this section, we prove the connection between perimeters of convex bodies in a partition and distances from vertices of this partition.

\begin{lemma}\label{lem:partition}
If a convex body with boundary $C$ is partitioned into convex bodies with boundaries $C_i$ with all partition vertices $v_j$ of degree 3, then for any convex curve $D$,
$$\sum_i \pper_D(C_i) = \pper_D(C)+\sum_j \pdist(v_j, D).$$
\end{lemma}

\begin{proof}
There are countably many values of $\theta$ such that $C$ or one of $C_i$ has a supporting line parallel to $u_{\theta}$ with more than one common point with it. It is, therefore, sufficient to check that for all other $\theta$, the integrand is the same in both sides of the suggested equality. Integrands in both sides contain distances from lines parallel to $u_{\theta}$ to $\ell_{\theta}(D)$. We just need to carefully check that the lines are the same in the right and the left hand sides.

First, we take into account lines that do not go through vertices of the partition. They show up as supporting lines of $C$ in the right hand side and then they show up in the left hand side too, as supporting lines of $C_i$.

All the remaining lines are going through vertices of the partition. There are two possible scenarios: all partition edges from a vertex $v_j$ are on one side with respect to a line parallel to $u_{\theta}$ or three edges are split into two non-empty groups. In the former case, $v_j$ is necessarily in $C$ and the line through $v_j$ contributes to the left hand side twice, in two $\pper(C_i)$, and to the right hand side twice, in $\pper(C)$ and one $\pdist(v_j)$. In the latter case, the line through $v_j$ contributes only to one of $\pper(C_i)$ and to one $\pdist(v_j)$.
\end{proof}

\section{Total perimeter of convex bodies}

In this section, we provide a short proof of Theorem \ref{thm:main}. For the first step of the proof, we use the structural result from the paper of Pinchasi.

\begin{proposition}\label{prop:pinchasi}\cite{P17}
For disjoint convex bodies $C_1$, $C_2$,$\ldots$, $C_k$ inside the convex body $D$, there exist a partition of $D$ into convex bodies $C_1'$, $C_2'$, $\ldots$, $C_{k+l}'$, $l\geq 0$, such that $C_i\subseteq C_i'$ for all $1\leq i\leq k$, and all $C_j'$, $k+1\leq j\leq k+l$, have no common points with each other and with the boundary of $D$.
\end{proposition}

Sets $C_j'$, $k+1\leq j\leq k+l$, are called holes of the partition. Note that $\sum_1^k \per(C_i')\geq \sum_1^k \per(C_i)$ so it is sufficient to prove the bound of Theorem \ref{thm:main} for the extended case only.

For the next step, we explain how to extend this partition to a partition of a constant width body $D'$, $D\subseteq D'$. We will need the following lemma.

\begin{lemma}\label{lem:AK}\cite[Lemma 4]{AK12}
Any finite convex partition of a convex body in $\mathbb{R}^2$ can be extended to a convex partition of $\mathbb{R}^2$.
\end{lemma}

We extend the partition of $D$ to the partition $\tilde{C}_1$, $\ldots$, $\tilde{C}_{k+l}$ of the plane by Lemma \ref{lem:AK} and use it to get a partition of $D'$ into $C_1''=\tilde{C}_1\cap D'$, $\ldots$, $C_{k+l}''=\tilde{C}_{k+l}\cap D'$. Note that this extension does not change any convex parts that were strictly inside $D$ including all holes. Let $l_1$, $l_2$, $\ldots$, $l_n$ be the lengths of new line segments added to the extended parts of the partition (see Figure \ref{fig:extension}). Then
$$\sum\limits_{i=1}^k \per(C_i') - \per(D) = \sum\limits_{i=1}^k \per(C_i'') - \per(D') - 2\sum\limits_{j=1}^n l_j.$$
If $D'$ has the same diameter as $D$, then it is sufficient to prove the bound for $D'$. Any convex body can be extended to a body of constant width \cite[Theorem 54]{E58} so it is sufficient for us to prove the bound of Theorem \ref{thm:main} for the case when the large convex body has constant width.

\begin{figure}[h]
\includegraphics[height=6cm]{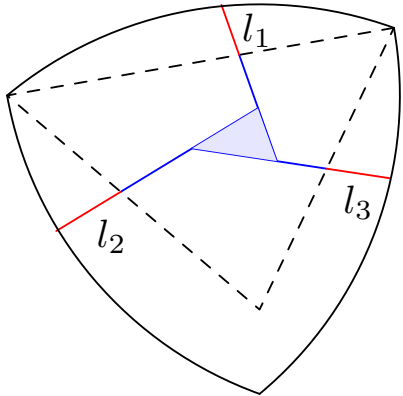}
\caption{Extending a partition of a triangle to the Reuleaux triangle of the same diameter}
\label{fig:extension}
\end{figure}

Boundaries of the partition sets form a plane graph. By adding vertices and introducing degenerate edges and faces we can assume that all vertices in this graph have degree 3. In the first scenario, assume there is a vertex $p$ of the plane graph surrounded by non-holes and connected to vertices $p_1$, $\ldots$, $p_r$, consecutively. Then we can introduce a degenerate hole $p_1'\ldots p_r'$ (all vertices of the hole geometrically coincide with $p$ and are distinguishable as planar graph vertices only) with degenerate edges $p_1'p_2'$, $\ldots$, $p_r'p_1'$ so that $p_1p_1'$, $\ldots$, $p_rp_r'$ are also edges of the plane graph. In the second scenario, assume there is a vertex $p$ of degree at least 4 on the boundary of the body $D$ or $p$ is a vertex of one of the holes. Let $p$ be connected to $p_1$, $\ldots$, $p_r$, consecutively, with $pp_1$ and $pp_r$ on the boundary of $D$/hole. Then we can add vertices $p_1',\ldots, p_r'$ (again geometrically coinciding with $p$) with degenerate edges $p_1'p_2'$, $\ldots$, $p_{r-1}'p_r'$ lying consecutively on the boundary of $D$/hole so that $p_1p_1'$, $\ldots$, $p_rp_r'$ are also edges of the plane graph.

After all reductions, Theorem \ref{thm:main} follows from the following theorem.

\begin{theorem}\label{thm:mainp}
A constant width body $D$ is partitioned into convex bodies $C_1$, $\ldots$, $C_{k}$, $H_1$, $\ldots$, $H_l$ such that all $H_j$, $1\leq j\leq l$, are pairwise disjoint and have no common points with the boundary of $D$ and all vertices of the graph defined by the partition have degree 3. Then
$$\sum\limits_{i=1}^k \pper(C_i)\leq \per(D)+2(k-1) \diam(D).$$
\end{theorem}

Before proving Theorem \ref{thm:mainp} we give a proof to the key lemma.

\begin{lemma}\label{lem:key}
For any convex polygon $P$ with vertices $a_i$, $1\leq i\leq m$, inside a constant width body $D$,
$$\sum_1^m \pdist(a_i, D)\leq \pper_D(P)+(m-2)\diam(D).$$ 
\end{lemma}

\begin{proof}
First, we show by induction that it is sufficient to prove this lemma for triangles. Indeed, if we already know that the statement holds for all polygons with not more than $m$ sides, $m\geq 3$, then we can partition an $m$-gon $P$ into an $(m-1)$-gon $P_1$ with the set of vertices $S_1$ and a triangle $P_2$ with the set of vertices $S_2$. By the inductive hypothesis, $$\sum_{a\in S_1} \pdist(a, D) + \sum_{a\in S_2} \pdist(a,D) \leq \pper_D(P_1)+(m-3)\diam(D)+\pper_D(P_2)+\diam(D).$$
By Lemma \ref{lem:partition},
$$\pper_D(P_1)+\pper_D(P_2)=\pper_D(P) + \sum_{a\in S_1\cap S_2}\pdist(a,D).$$
Substituting this sum in the inequality above we get the required bound for the $m$-gon $P$.

For a triangle $a_1a_2a_3$,
\begin{equation}\label{eqn:mid}
\sum_1^3 \pdist(a_i, D)- \pper_D(a_1a_2a_3)=\int\limits_0^{\pi} \dist(\ell_{\theta}^{mid}, \ell_{\theta}(D)) d\theta,
\end{equation}
where $\ell_{\theta}^{mid}$ is the interjacent line of the triangle in direction of $u_\theta$, that is, $\ell_{\theta}^{mid}$ goes through a vertex and crosses the triangle. If $a_1a_2a_3$ is a degenerate triangle, one point or a line segment, the integral above is the same as $\pdist(a,D)$ for one of the vertices $a$ so it is not larger than $\diam(D)$ by Lemma \ref{lem:pdist_diam}.

For nondegenerate triangles, the convexity of the integral in (\ref{eqn:mid}) (the proof of convexity is analogous to the proof of Proposition \ref{prop:pseudometric}) implies that among all triangles with parallel sides, the maximum of $\sum \pdist(a_i, D)$ is necessarily attained on a triangle with at least two vertices on the boundary of $D$.

Each curve of constant width can be approximated by a smooth curve of constant width (see, for instance, \cite{T76, W77}) so we assume that $D$ has a smooth boundary. We can also assume that the curvature of the boundary is bounded from above (for instance, by approximating $D$ with $D+\epsilon \omega$, where $\omega$ is a circle with center at the origin and unit radius). The function $\pdist(a_1,D)+\pdist(a_2,D)+\pdist(a_3,D) - \pper_D(a_1a_2a_3)$ is continuous with respect to $a_1$, $a_2$, $a_3 \in D$ so it reaches its maximum at some triple of points. We consider the triple where it is maximal and, using the argument from above, assume $a_1$ and $a_2$ are on the boundary of $D$. We also assume that the boundary of $D$ is parametrized by $\gamma(\theta)$.

For the next step of the proof, we show that normal lines trough $a_1$ and $a_2$ are necessarily interior angle bisectors of $\triangle a_1a_2a_3$. Assume the normal line at $a_1$ is not the angle bisector of $\triangle a_1a_2a_3$.
The idea of the proof is to move $a_1$ along the curve by a small distance $\Delta$ and show that the integral can increase with this move. When doing so, all interjacent lines $\ell_{\theta}^{mid}$ through $a_1$ change linearly in terms of $\Delta$ and we will show that this change entails the linear part of the change in \ref{eqn:mid}. At the same time, the measure of $\theta$ such that lines $\ell_{\theta}^{mid}$ change a vertex they pass through is also of linear size in terms of $\Delta$ so in total the change of this kind gives only an $O(\Delta^2)$ error. Overall, the change is linear in $\Delta$ so we can move $a_1$ to increase the integral.

Let us write these arguments in more detail. Denote by $\ell_{\theta}^1$, $\ell_{\theta}^2$, $\ell_{\theta}^3$ the lines with direction of $u_\theta$ through $a_1$, $a_2$, $a_3$, respectively. Assume $\ell_{\theta}^1$ is an interjacent line when $\theta\in[\theta_1,\theta_2]$, $\ell_{\theta}^2$ is an interjacent line when $\theta\in[\theta_2,\theta_3]$, and $\ell_{\theta}^3$ is an interjacent line when $\theta\in[\theta_3,\theta_1+\pi]$. Also assume $\ell_{\theta^*}$ is a normal line through $a_1$ and $\theta^*\in(\theta_1,\theta_2)$ but $\theta^*\neq (\theta_1+\theta_2)/2$ (the case when $\theta^*\notin(\theta_1,\theta_2)$ will be considered later). Under this notation,
$$\int\limits_0^{\pi} \dist(\ell_{\theta}^{mid}, \ell_{\theta}(D)) d\theta = \int\limits_{\theta_1}^{\theta_2} \dist(\ell_{\theta}^1, \ell_{\theta}(D)) d\theta + \int\limits_{\theta_2}^{\theta_3} \dist(\ell_{\theta}^2, \ell_{\theta}(D)) d\theta + \int\limits_{\theta_3}^{\theta_1+\pi} \dist(\ell_{\theta}^3, \ell_{\theta}(D)) d\theta.$$

Now we vary $a_1$ by moving it along the boundary of $D$ to a point $a_1'$ so that $|a_1a_1'|=\Delta$ and denote by $\ell_{\theta}^{1'}$ a line in direction of $u_{\theta}$ through $a_1'$. At this point our goal is to show that the change of $\int_0^{\pi} \dist(\ell_{\theta}^{mid}, \ell_{\theta}(D)) d\theta$ is linear with respect to $\Delta$. First, we denote new interjacent lines by $\ell_{\theta}^{mid'}$ and note that $\dist(\ell_{\theta}^{mid}, \ell_{\theta}(D)) - \dist(\ell_{\theta}^{mid'}, \ell_{\theta}(D))\in O(\Delta)$ for all $\theta$. Second, we assume that new angles where interjacent lines change are $\theta_1'$, $\theta_2'$, and $\theta_3'$ and note that $\theta_i-\theta_i'\in O(\Delta)$ for $i=1,2,3$. Then

$$\int\limits_0^{\pi} \dist(\ell_{\theta}^{mid'}, \ell_{\theta}(D)) d\theta = \int\limits_{\theta_1'}^{\theta_2'} \dist(\ell_{\theta}^{1'}, \ell_{\theta}(D)) d\theta + \int\limits_{\theta_2'}^{\theta_3'} \dist(\ell_{\theta}^2, \ell_{\theta}(D)) d\theta + \int\limits_{\theta_3'}^{\theta_1'+\pi} \dist(\ell_{\theta}^3, \ell_{\theta}(D)) d\theta$$
$$=\int\limits_{\theta_1}^{\theta_2} \dist(\ell_{\theta}^{1'}, \ell_{\theta}(D)) d\theta + \int\limits_{\theta_2}^{\theta_3} \dist(\ell_{\theta}^2, \ell_{\theta}(D)) d\theta + \int\limits_{\theta_3}^{\theta_1+\pi} \dist(\ell_{\theta}^3, \ell_{\theta}(D)) d\theta + O(\Delta^2)$$
so the linear part of the change may show up only in the first integral.

Using Lemma \ref{lem:signed} we get
$$\int\limits_{\theta_1}^{\theta_2} \dist(\ell_{\theta}^{1}, \ell_{\theta}(D)) d\theta = \int\limits_{\theta_1}^{\theta^{*}} (a_1-\gamma(\theta))\cdot v_{\theta}\, d\theta + \int\limits_{\theta^{*}}^{\theta_2} (\gamma(\theta)-a_1)\cdot v_{\theta}\, d\theta$$
$$=(a_1-\gamma(\theta_1))\cdot u_{\theta_1} - 2(a_1-\gamma(\theta^*))\cdot u_{\theta^*} + (a_1-\gamma(\theta_2))\cdot u_{\theta_2}$$
$$=a_1\cdot(u_{\theta_1}-2u_{\theta^*}+u_{\theta_2}) -\gamma(\theta_1)\cdot u_{\theta_1}+2\gamma(\theta^*)\cdot u_{\theta^*}-\gamma(\theta_2)\cdot u_{\theta_2}.$$														

Assume $\ell_{\theta^{**}}$ is the normal line at $a_1'$, $\theta^{**}\in(\theta_1,\theta_2)$. Note that $\theta^{**}-\theta^{*}\in O(\Delta)$ because the curvature of the boundary of $D$ is bounded from above. Then
$$\int\limits_{\theta_1}^{\theta_2} \dist(\ell_{\theta}^{1'}, \ell_{\theta}(D)) d\theta = \int\limits_{\theta_1}^{\theta^{**}} (a_1'-\gamma(\theta))\cdot v_{\theta}\, d\theta + \int\limits_{\theta^{**}}^{\theta_2} (\gamma(\theta)-a_1')\cdot v_{\theta}\, d\theta$$
$$ = \int\limits_{\theta_1}^{\theta^{*}} (a_1'-\gamma(\theta))\cdot v_{\theta}\, d\theta + \int\limits_{\theta^{*}}^{\theta_2} (\gamma(\theta)-a_1')\cdot v_{\theta}\, d\theta + O(\Delta^2)$$
$$=a_1'\cdot(u_{\theta_1}-2u_{\theta^*}+u_{\theta_2}) -\gamma(\theta_1)\cdot u_{\theta_1}+2\gamma(\theta^*)\cdot u_{\theta^*}-\gamma(\theta_2)\cdot u_{\theta_2} + O(\Delta^2).$$														
$$=\int\limits_{\theta_1}^{\theta_2} \dist(\ell_{\theta}^{1}, \ell_{\theta}(D)) d\theta + (a_1'-a_1)\cdot (u_{\theta_1}-2u_{\theta^*}+u_{\theta_2})  + O(\Delta^2).$$
Due to the smoothness of the boundary of $D$, $a_1'-a_1=\pm\Delta\, v_{\theta^*} + O(\Delta^2)$ so
$$\int\limits_{\theta_1}^{\theta_2} \dist(\ell_{\theta}^{1'}, \ell_{\theta}(D)) d\theta - \int\limits_{\theta_1}^{\theta_2} \dist(\ell_{\theta}^{1}, \ell_{\theta}(D)) d\theta = \pm\Delta\, v_{\theta^*} \cdot (u_{\theta_1}-2u_{\theta^*}+u_{\theta_2})  + O(\Delta^2).$$
The dot product $v_{\theta^*} \cdot (u_{\theta_1}-2u_{\theta^*}+u_{\theta_2})$ is not zero since $\theta^*\neq (\theta_1+\theta_2)/2$ so there is a linear component in the change of the integral. We choose the direction of the move depending on the sign of the dot product to ensure that the difference is positive and the integrals increases. Overall, we conclude the maximal value cannot be attained on a triangle unless the normal line through $a_1$ is the interior angle bisector of $\triangle a_1a_2a_3$. 

In the case $\theta^*<\theta_1$, there is a direction along the boundary of $D$ such that moving $a_1$ in this direction increases all $\dist(\ell_{\theta}^1,\ell_{\theta}(D))$ linearly in $\Delta$. This move entails a linear increase of the integral from (\ref{eqn:mid}). The other cases for $\theta^*$ are analogous with additional $O(\Delta^2)$ terms for $\theta^*=\theta_1$ or $\theta_2$ (because $\theta^{**}$ may jump inside the interval $(\theta_1,\theta_2)$).

We have showed that the normal line at $a_1$ is the angle bisector of the triangle. Analogously, the normal line at $a_2$ must be the interior angle bisector as well.

For the last step of the proof, we will use the following geometric fact by Balitskiy.
\begin{lemma}\label{lem:bal}\cite[Proof of Theorem 4.1]{Bal16}
Let points $y_1$ and $y_2$ be chosen on the interior angle bisectors of $\triangle x_1x_2x_3$ from $x_1$ and $x_2$, respectively. If $|x_1y_1|=|x_2y_2|\geq \frac 1 2 \per(x_1x_2x_3)$, then $|y_1y_2|>|x_1y_1|$.
\end{lemma}

Now we choose the points $b_1$ and $b_2$ in $D$ diametrically opposite to $a_1$ and $a_2$, respectively. We know that $|a_1b_1|=|a_2b_2|=\diam(D)$ and $|b_1b_2|<\diam(D)$. By Lemma \ref{lem:bal}, $\per(a_1a_2a_3)\geq 2|a_1b_1|=2\diam(D)$. Then, as required,
$$\sum_1^3 \pdist(a_i, D)- \pper_D(a_1a_2a_3)\leq 3\diam(D)-\per(a_1a_2a_3) \leq \diam(D).$$
\end{proof}

\begin{proof}[Proof of Theorem \ref{thm:mainp}]
By Euler's formula, the number of partition vertices is $2(k+l-1)$. Then $$\sum_{i=1}^k \pper_D(C_i) = \per(D)+\sum_{j=1}^{2(k+l-1)} \pdist(v_j, D)- \sum_{i=1}^l \pper_D(H_i)$$
Using Lemma \ref{lem:key} for all holes we get
$$\sum_{i=1}^k \pper_D(C_i) \leq \per(D)+\sum_{v_j\notin\cup H_i} \pdist(v_j, D)+\sum_{v_j\in \cup H_i} \diam(D) - 2 l\ \diam(D)$$
Finally, by Lemma \ref{lem:pdist_diam}, all $\pdist(v_j, D)\leq \diam(D)$, so
$$\sum_{i=1}^k \pper_D(C_i)\leq \per(D) + 2(k+l-1) \diam (D) - 2l\ \diam(D)$$
$$=\per(D)+2(k-1)\diam(D).$$
\end{proof}

\section{Acknowledgments}

The authors would like to thank Filip Mori\'{c}, Rom Pinchasi, Ilya Bogdanov, and Sergei Tabachnikov for useful comments and discussions that eventually led to the current text. A.G. was partially supported by the NSF grant DMS-2054536.

\bibliographystyle{amsalpha}

\begin{thebibliography}{A}

\bibitem{AK12}
A.~Akopyan and R.~Karasev, \emph{Kadets-Type Theorems for Partitions of a Convex Body}. Discrete Comput. Geom. 48 (2012), 766--776.

\bibitem{A39}
A.~D.~Aleksandrov. \emph{Almost everywhere existence of the second differential of a convex function and some properties of convex surfaces connected with it} (in Russian). Uch. Zap., Leningrad. Gos. Univ. Math. Ser. 6, 3--35 (1939)

\bibitem{AK52}
R. D. Anderson and V. L. Klee, Jr., \emph{Convex functions and upper semi-continuous collections}. Duke Math. J. 19 (1952), 349--357.

\bibitem{Bal16}
A.~Balitskiy, \emph{Shortest closed billiard trajectories in the plane and equality cases in Mahler’s conjecture}. Geom. Dedicata. 184 (2016), 121–-134.

\bibitem{E58}
H.~G.~Eggleston. \emph{Convexity.} Cambridge Tracts in Mathematics and Mathematical Physics, No. 47, Cambridge University Press, New York, 1958.

\bibitem{GM14}
A. Glazyrin and F. Mori\'{c}. \emph{Upper bounds for the perimeter of plane convex bodies}, Acta Math. Hungar. 142 (2014), 366--383.

\bibitem{HS53I}
P.~C.~Hammer and A.~Sobczyk, \emph{Planar line families. I.} Proceedings of the American Mathematical Society 4, no. 3 
(1953), 226--233.

\bibitem{HS53II}
P.~C.~Hammer and A.~Sobczyk, \emph{Planar line families. II.} Proceedings of the American Mathematical Society 4, no. 3 
(1953), 341--349.

\bibitem{MMO19}
H.~Martini, L.~Montejano, and D.~Oliveros. \emph{Bodies of constant width}. Springer International Publishing, 2019.

\bibitem{P17}
R.~Pinchasi, \emph{On the perimeter of k pairwise disjoint convex bodies contained in a convex set in the plane.} Combinatorica 37 (2017), 99--125.

\bibitem{T76}
Sh.~Tanno, \emph{$C^{\infty}$-approximation of continuous ovals of constant width.} J. Math. Soc. Japan Vol. 28, No. 2 (1976), 384--395.

\bibitem{W77}
B.~Wegner, \emph{Analytic approximation of continuous ovals of constant width.} J. Math. Soc. Japan Vol. 29, No. 3 (1977), 537--540.


\end{thebibliography}

\end{document}